\documentclass[11pt,english,oneside]{amsart}
\usepackage[T1]{fontenc}
\usepackage[latin9]{inputenc}
\usepackage{geometry}
\usepackage{amssymb}
\usepackage{color}

\usepackage{graphicx,color}
\usepackage{amsmath, amssymb, graphics}

\usepackage{graphicx}

\makeatletter
\numberwithin{equation}{section} %% Comment out for sequentially-numbered
\numberwithin{figure}{section} %% Comment out for sequentially-numbered
  \@ifundefined{theoremstyle}{\usepackage{amsthm}}{}
  \theoremstyle{plain}
  \newtheorem{thm}{Theorem}[section]
  \theoremstyle{plain}
  
  \theoremstyle{plain}
  
  \theoremstyle{Remark}
  \newtheorem{rem}[thm]{Remark}
  \theoremstyle{remark}
  
  \theoremstyle{plain}
  \newtheorem{lem}[thm]{Lemma}

\usepackage{geometry}

\smallskip
\def\<{{\langle }}
\def\>{{\rangle }}

\makeatother

\usepackage{babel}

\smallskip
\def\<{{\langle }}
\def\>{{\rangle }}

\makeatother

\begin{document}

\title[spectrum rotational cmc]{Spectrum of the Laplacian and the Jacobi operator on rotational cmc hypersurfaces of spheres}

\author{ Oscar M. Perdomo }

\date{\today}

%\curraddr{Department of Mathematics\\
%Central Connecticut State University\\
%New Britain, CT 06050\\
%}

%\email{ perdomoosm@ccsu.edu}


\maketitle

%ABSTRACT

\begin{abstract}
Let $M\subset \mathbb{S}^{n+1}\subset\mathbb{R}^{n+2}$ be a compact cmc rotational hypersurface of the $(n+1)$-dimensional Euclidean unit sphere.  Denote by $|A|^2$ the square of the norm of the second fundamental form and $J(f)=-\Delta f-nf-|A|^2f$ the stability or Jacobi operator.  In this paper we compute the spectra of their Laplace and Jacobi operators in terms of eigenvalues of second order Hill's equations. 
 For the minimal rotational examples, we prove that the stability index --the numbers of negative eigenvalues of the Jacobi operator counted with multiplicity -- is greater than $3 n+4$ and we also prove that there are at least 2 positive eigenvalues of the Laplacian of $M$ smaller than $n$. When $H$ is not zero, we have that
every non-flat CMC rotational immersion is generated by rotating a planar profile curve along a geodesic called the axis of rotation. Let $m$ be the number of points where the maximal distance from this profile curve to the origin is achieved (we assume that the coordinates of the plane containing the profile curve has been set up so that the axis of rotation goes through the origin). Let $l$ be be the wrapping number of the profile curve. We show that the number of negative eigenvalues of the operator  $J$ counted with multiplicity is at least $(2l-1)n+(2m-1)$. This result was proven for the case $n=2$ by Rossman and Sultana. They called $m$ the number of bulges or the number of necks. We will slightly change the definition of $l$  to include immersed examples that contain the axis of rotation.
\end{abstract}

%INTRODUCTION
\section{Introduction}

Rotational constant mean curvature hypersurfaces of spheres provide a variety of examples that can be used to understand the nature of cmc hypersurfaces in general. In this paper we derive a formula for the Laplacian on these hypersurfaces that allows us to compute the spectra of their Jacobi and Laplace operators to any desired degree of accuracy. Since we are using minimal hypersurfaces to motivate our results, here in the introduction we denote by $\mathcal{M}^n$ any minimal, not necessarily rotational, compact $n$-dimensional minimal hypersurface of the sphere. The spectra of the Laplacian and the stability operators on compact minimal hypersurfaces $\mathcal{M}^n\subset \mathbb{S}^{n+1}$ have been one of the central topics in differential geometry. Let us denote by $\lambda_1(\mathcal{M}^n)$, the first nonzero eigenvalue of the Laplace operator of $\mathcal{M}$. For minimal hypersurfaces of spheres, the spectrum of the Laplacian is known only when $\mathcal{M}^n$ is an Euclidean sphere,  $\mathcal{M}^n$ is the product of Euclidean spheres, or $\mathcal{M}^n$ is a cubic isoparametric hypersurface, \cite{B1}, \cite{B2}. It is known that for any minimal hypersurface of the sphere, $n$ is one of the eigenvalues of the Laplacian. Yau has conjectured that if $\mathcal{M}^n$ is compact and embedded, then $\lambda_1(\mathcal{M}^n)=n$,  \cite{Y}. A positive partial result of this conjecture was given by Choi and Wang \cite{CW}. They showed that  if $\mathcal{M}^n$ is compact and embedded, then $\lambda_1(\mathcal{M}^n)\ge \frac{n}{2}$. In this paper we prove that if $\mathcal{M}^n$ is rotational, then $\lambda_1(\mathcal{M}^n)<n$. Moreover we show that there are at least two positive eigenvalues of the Laplace operator smaller $n$.
Regarding the spectrum of the Jacobi operator, we have that the number of negative eigenvalues of the operator $J$ counted with multiplicity is known as the stability index of $\mathcal{M}^n$ and it is denoted as $\mathrm{ind}(\mathcal{M}^n)$. It has been conjectured that if $\mathcal{M}^n$ is not an Euclidean sphere, then $\mathrm{ind}(\mathcal{M}^n)=n+3$ implies that $\mathcal{M}^n$ is a product of Euclidean spheres.  This conjecture was proven by Urbano, \cite{U}, when $n=2$. For general $n$ only partial results are known. In \cite{P2}, the author shows that if $\mathrm{ind}(\mathcal{M}^n)=n+3$, then $\int_M|A|^2\le \int_M(n-1)$ with equality only if $\mathcal{M}^n$ is a product of Euclidean spheres. On the other hand in \cite{P4}, the author shows that if  $\mathcal{M}^n$ is rotational, then,  $\int_M|A|^2\le \int_M(n-1)$. Therefore, it may seem that the rotational minimal hypersurfaces are good candidates for a counterexample of the conjecture. This is not the case. Due to their symmetries, their stability index must be greater than $n+3$. See \cite{P1}. In this paper we show that there is a big jump in the stability index among the minimal rotational examples. We prove that if $\mathcal{M}^n$  is not a product of spheres then  $\mathrm{ind}(\mathcal{M}^n)\ge 3n+5$.  Some other partial results on this conjecture are found in \cite{S}, \cite{P3}, \cite{SA}. One of the most important applications of this type of estimates on the stability index was given  by Marques and Neves, \cite{MN}, where, among other tools, they used Urbano's result to prove Willmore's conjecture. 

To illustrate the way we can use our results to estimate the eigenvalues, we pick a $3$-dimensional rotational minimal hypersurface in $\mathbb{S}^4$ and prove that the first three eigenvalues of the Laplace operator are: 0, a number near $0.4404$ with multiplicity 2, and 3 with multiplicity 5. We also show that the negative eigenvalues of the Jacobi operator are: a number near $-8.6534$ with multiplicity 1, a number near $-8.52078$ with multiplicity 2, $-3$ with multiplicity 5, a number near  -2.5596 with multiplicity 6, and  a number number near $-1.17496$ with multiplicity 1. The stability index of this hypersurface is thus 15.

When $H$ is not zero, we prove (See Theorem \ref{indexcmc}) a lower bound for the number of negative eigenvalues of the Jacobi operator that generalizes to any dimension the result on the morse index of constant mean curvature tori of revolution in the 3-sphere proven by Rossman and Sultana in \cite{RS}.

We will be using the {\it Oscillation Theorem} for the periodic problem on the Hill's equation (a proof can be found in \cite{MW}). 

\begin{thm}\label{ot}
Consider the  differential equation 

\begin{eqnarray}\label{theode}
z^{\prime\prime}(t)+(\lambda+Q(t))z(t)=0
\end{eqnarray}

where $Q$ is a smooth $T$-periodic function. For any $\lambda$ let us define 
 
$$\delta(\lambda):= z_1(T,\lambda)+z_2^\prime(T,\lambda)$$
 
where $z_1(t,\lambda)$ and $z_2(t,\lambda)$ are solutions of (\ref{theode}) such that $z_1(0,\lambda)=1$, $z_1^\prime(0,\lambda)=0$ and $z_2(0,\lambda)=0$, $z_2^\prime(0,\lambda)=1$. There exists an increasing infinite sequence of real numbers $\lambda_1,\lambda_2\dots$ such the differential equation (\ref{theode}) has a $T$-periodic solution if and only if $\lambda=\lambda_j$. Moreover the $\lambda_j$ are the roots of the equation $\delta(\lambda)=2$. $\delta$ is called the discriminant function of the operator $K[z]=z^{\prime\prime}(t)+Q(t)z(t)$.
\end{thm}

We will also be using the following theorem proven by Haupt (1914) from the Hill's equation theory. The next presentation of Haupt theorem can be found in \cite{MW}.

\begin{thm} \label{zeroes} Let us denote by $\lambda_1<\lambda_2\le \lambda_3\le \lambda_4\dots$ the sequence of eigenvalues of the Hill's equation presented in Therem \ref{ot}. If $z(t)$ is a nonzero $T$-periodic solution of the equation (\ref{theode}) with $\lambda=\lambda_i$ then, the number of zeros of $z(t)$ in the interval $[0,T) $ is $2 \lfloor \frac{i}{2}\rfloor$.

\end{thm}

We would like to point out that Beeckmann and Lokes have used the Hill equation to find bounds on the eigenvalues of the Laplacian on toroidal surfaces, \cite{BL}

The author would like to thank Andr\'es Rivera, Bruce Solomon and Nelson Casta\~neda for their useful comments and suggestions.

%SECTION 2
%DESCRIBIN ROTATIONAL CMC HYPERSURFACES

\section{Describing rotational cmc hypersufaces of spheres} Any compact cmc rotational hypersurface of $\mathbb{S}^{n+1}$ is given by an immersion $\phi:\mathbb{S}^{n-1}\times \mathbb{R}\to \mathbb{S}^{n+1}$ where,

\begin{eqnarray}\label{theimm}
\phi(y,t)=(r(t)\, y,\sqrt{1-r(t)^2} \cos(\theta(t)),\sqrt{1-r(t)^2} \sin(\theta(t)))
\end{eqnarray}

and $r(t)$ is positive $T$-periodic function that satisfies the following conditions

\begin{eqnarray}\label{sode}
(r^\prime)^2+r^2(1+\lambda^2)=1\, ,
\end{eqnarray}

with

\begin{eqnarray}\label{rel}
\lambda=H+c^{-n/2} r^{-n}, \quad \theta(t)=\int_0^t\frac{r(\tau)\lambda(\tau)}{1-r^2(\tau)}\, d\tau\, ,
\end{eqnarray}

%PERIODICITY CONDITION

and $c$ is a positive real number that satisfies that 
\begin{eqnarray}\label{p}
\theta(T)=2\pi \frac{l}{m} \quad\hbox{where $l$ and $m$ are relative prime integers.}
\end{eqnarray}

The condition on $c$ in equation (\ref{p}) guarantees that the immersion $\phi$ satisfies $\phi(y,t+mT)=\phi(y,t)$ and makes $M$  compact. Recall that the function $r(t)$ depends on $c$ since $\lambda(t)$ depends on $c$. Also, since the function $\theta(t)$ depends on $r(t)$ and $\lambda(t)$, then $\theta(T)$ depends on $c$ as well.

%Otsuki remark
\begin{rem} \label{or} When $M$ is minimal, Otsuki \cite{O1},  \cite{O2}, showed that the expression

$$\theta(T)=K(c)=\int_0^\frac{T}{2}\frac{c^{-1/n}\, r^{1-n}(\tau)}{1-r^2(\tau)}\, d\tau$$

lies between lies $\pi$ and $\sqrt{2} \pi$. As a consequence the equation $\theta(T)=2\pi \frac{l}{m}$  cannot be solved with $l = 1$ and therefore no minimal rotational hypersurface is embedded.
\end{rem}

The principal curvatures of $M$ are $\lambda$ with multiplicty $(n-1)$ and $\mu=H-(n-1)c^{-n/2} r^{-n}$ with multiplicity one. Differentiating equation (\ref{sode}) we obtain

\begin{eqnarray}\label{soder}
 r^{\prime\prime}+r+r\lambda\mu=0 \, .
\end{eqnarray}

The next expression explicitly provides the Gauss map $\nu$ of the immersion (\ref{theimm}),

\begin{eqnarray}\label{thegauss}
\nu(y,t)=(-r\lambda\, y,\frac{r^2\lambda  \cos\theta-r^\prime  \sin\theta}{\sqrt{1-r(t)^2}},\frac{r^2\lambda  \sin\theta+r^\prime  \cos\theta}{\sqrt{1-r(t)^2}})
\end{eqnarray}

All the details of the construction of these hypersurfaces can be found in \cite{P}.

%SECTION 3

\section{Main theorems}

Before stating the following theorem, we recall that $m$ is the integer given in equation (\ref{p}),  $T$ is the period of the function $r(t)$, and $mT$ is the period of the immersion $\phi$.

%EXPRESSION FOR THE LAPLACIAN

\begin{thm} \label{lap} Let $M$ be a rotation hypersurface defined by equation (\ref{theimm}). For any function $\bar{f}:\mathbb{S}^{n-1}\longrightarrow \mathbb{R}$ we define $f:M\longrightarrow \mathbb{R}$ as $f(\phi(t,y))=\bar{f}(y)$. Likewise, for any $mT$-periodic function $\bar{g}:\mathbb{R}\longrightarrow \mathbb{R}$ we define $g:M\longrightarrow \mathbb{R}$ as $g(\phi(t,y))=\bar{g}(t)$. We will denote by $\bar{\Delta}$ the Laplacian operator on $\mathbb{S}^{n-1}$. With this notation we have  
\begin{eqnarray}\label{lap}
\Delta (fg)=f\left(\bar{g}^{\prime\prime}+(n-1)\frac{r^\prime}{r} \bar{g}^\prime\right)+\frac{\bar{\Delta}(\bar{f})}{r^2}\,  g\, .
\end{eqnarray}
\end{thm}

\begin{proof}
The proof is a direct computation using the fact that $\nabla (fg)=\frac{g}{r}\bar{\nabla} \bar{f}+f\bar{g}^\prime \frac{\partial}{\partial t}$ and that $\mathrm{div}(\frac{\partial}{\partial t})=(n-1)\frac{r^\prime}{r}$.

\end{proof} 

\begin{rem} 

We would like to point out that we are using the fact that the ambient space is $\mathbb{R}^{n+2}$ in the argument above. In particular we are using the natural identification of all tangent spaces $T_x \mathbb{R}^{n+2}$ with $\mathbb{R}^{n+2}$. If we decide to work intrinsically we will notice that the formula that compares the gradients will have an additional factor of $r$.  From the point of view of intrinsic differential geometry, the formula in Theorem \ref{lap} can be generalized to wraped products. See \cite{MG}.

\end{rem}

%COMPARING EIGENVALUES OF M WITH EIGENVALUES OF HILL'S EQUATIONS

\begin{thm}\label{mt}
Let $\alpha_1=0,\alpha_2=(n-1),\dots, \alpha_{k}=(k-1)(n+k-3)\dots$  denote the spectrum of $\mathbb{S}^{n-1}$. The spectrum of the Laplace operator $\Delta$ on $M$ is  given by $\cup_{k=1}^\infty \Gamma_k$, where,

$$\Gamma_k=\{ \lambda(k,1),\lambda(k,2),\dots \}$$

is the  ordered spectrum of the operator 

$$K_{\Delta,k}[z]= z^{\prime\prime}+(n-1)\frac{r^\prime}{r} z^\prime -\frac{\alpha_k}{r^2}\,  z $$

The spectrum of the Jacobi operator $J$ on $M$ is given by $\cup_{k=1}^\infty \mathbb{F}_k$, where

$$\mathbb{F}_k=\{ \tilde{\lambda}(k,1),\tilde{\lambda}(k,2),\dots \}$$

is the ordered spectrum of the operator 

$$K_{J,k}[z]= z^{\prime\prime}+(n-1)\frac{r^\prime}{r} z^\prime 
+\left(n+ nH^2+n(n-1) c^{-n} r^{-2 n}-\frac{\alpha_k}{r^2} \right)\,  z $$

\end{thm}

\begin{proof}
Let us prove the case of the Laplacian. By Theorem \ref{lap} we have that if  $\bar{f_k}$ is an eigenfunction of the Laplacian on $\mathbb{S}^{n-1}$ with eigenvalue $\alpha_k$ and $\bar{g}_l(t)$ is an eigenfunction of the operator $K_{\Delta,k}$ with eigenvalue $\lambda(k,l)$, then $h_{k,l}:M\longrightarrow \mathbb{R}$ given by $h_{kl}=f_kg_l$ is an eigenfunction of the Laplacian with  eigenvalue $\lambda(k,l)$. Therefore $\cup_{k=1}^\infty \Gamma_k$ is contained in the spectrum of $\Delta$.  To prove  the reverse inclusion, we only need to point out that every smooth function $h:M\longrightarrow\mathbb{R}$ can be written as a series of eigenfunctions of the form $h_{kl}$.  There exists a basis

$$\bar{f}_{1,1}, \bar{f}_{2,1}, \bar{f}_{2,2},\dots, \bar{f}_{2,n}, \bar{f}_{3,1}, \bar{f}_{3,2},\dots \bar{f}_{3,2n}, \bar{f}_{4,1},\dots $$

for $L^2(\mathbb{S}^{n-1})$ with $\bar{\Delta} \bar{f}_{k,j}+\alpha_k\bar{f}_{k,j}=0$. So any $h:M\longrightarrow\mathbb{R}$ can be written as a sum of the form

$$a_{1,1}(t) \bar{f}_{1,1}+ a_{2,1}(t)  \bar{f}_{2,1}+\dots + a_{2,n}(t)  \bar{f}_{2,n}+a_{3,1}(t)  \bar{f}_{3,1}+\dots$$

We obtain the desired expression for the function $h:M\longrightarrow\mathbb{R}$, by noticing that each $a_{k,l}(t)$ can now be expanded in eigenfunctions of the operator $K_{\Delta,k}$. The proof for the Jacobi operator is similar and uses the following expression for $|A|^2$,

$$|A|^2=n(H^2+(n-1)c^{-n}r^{-2n})$$

\end{proof}

%LEMMA THAT ALLOWS US WE CAN ADDAPT THE OSCILLATION THEOREM 

The following lemma allows us to use the Oscillation Theorem to compute the eigenvalues  for the second order differential equations on Theorem \ref{mt}.

\begin{lem} Let us denote by $\lambda=H+c^{-n/2} r^{-n}$ and $\mu=H-(n-1) c^{-n/2} r^{-n}$. The change of variables $u=r^{\frac{n-1}{2}}z$ gives us, 

\begin{eqnarray*}
K_{\Delta,k}&=& z^{\prime\prime}+(n-1)\frac{r^\prime}{r} z^\prime -\frac{\alpha_k}{r^2}\,  z \\
&=&\frac{1}{r^{\frac{n-1}{2}}}\left(  u^{\prime\prime}+ \left(\frac{1}{4} \lambda ^2 ((n-1) (n-3))+\frac{1}{2} \lambda  \mu  (n-1)-\frac{4 \alpha_k +(n-3) (n-1)}{4 r^2}+\frac{1}{4} (n-1)^2\right) u\right)
\end{eqnarray*}

and

\begin{eqnarray*}
K_{J,k}&=& z^{\prime\prime}+(n-1)\frac{r^\prime}{r} z^\prime 
+\left(n+ nH^2+n(n-1) c^{-n} r^{-2 n}-\frac{\alpha_k}{r^2} \right)\,  z \\
&=&\frac{1}{r^{\frac{n-1}{2}}}\left(  u^{\prime\prime}+ \frac{1}{4} \left(\lambda^2 \left(n^2-1\right)+2 \lambda \mu  (n-1)+4 \mu^2+(n+1)^2-\frac{4 \alpha_k +(n-3) (n-1)}{r^2}\right) u\right)
\end{eqnarray*}

\end{lem}

%ESTIMATES ON THE INDEX AND THE POSITION OF N AS AN EIGENVALUE OF THE LAPLACIAN.

\begin{thm}\label{index}
If $\mathcal{M}\subset \mathbb{S}^{n+1}$ is a rotational minimal compact hypersurface, then $\mathrm{ind}(\mathcal{M}^n)\ge 3n+5$ and there are at least two  positive eigenvalues of the Laplace operator smaller than $n$. More precisely, if $m$ and $l$ are the relative prime integers in equation (\ref{p}), then 
\begin{itemize}
\item
$l\ge2$, $m\ge 3$,
\item
 $\mathrm{ind}(\mathcal{M}^n)\ge (2l-1)n+(2m-1)$ and,
 \item
 there are exactly $2l-2$ positive eigenvalues of the Laplacian  smaller than $n$.
\end{itemize}
\end{thm}

\begin{proof}
As pointed out in Remark \ref{or}, the positive integer $l$ that satisfies  the equation $\theta(T)=2\pi \frac{l}{m}$ must be greater than $1$. We also have that $m\ge 3$. Recall that Otsuki have shown that, for any $c$, the integral that defines $\theta(T)$ is a number between $\pi$ and $\sqrt{2}\pi$. In this proof we will be using the notation for $\lambda(k,j)$ and $\tilde{\lambda}(k,j)$ introduced in Theorem \ref{mt}. More precisely, the eigenvalues of the operator $K_{\Delta,k}$ are 

$$\lambda(k,1)<\lambda(k,2)\le \lambda(k,3)\dots$$

and the eigenvalues of the operator $K_{J,k}$ are 

$$\tilde{\lambda}(k,1)<\tilde{\lambda}(k,2)\le \tilde{\lambda}(k,3)\dots$$

A direct verification shows that the functions $f_1(t)= \sqrt{1-r^2(t)}\cos(\theta)$ and $f_2(t)= \sqrt{1-r^2(t)}\sin(\theta)$ satisfy the equation $K_{\Delta,1}(f_i)+nf_i=0$, for $i=1,2$.  The functions $f_1$ and $f_2$ have  $2l$ zeroes in the interval $[0,mT)$  By Theorem \ref{zeroes} we conclude that $n=\lambda(1,2l)=\lambda(1,2l+1)$. Therefore, the eigenvalues of  $K_{\Delta,1}$ are

$$\lambda(1,1)=0<\lambda(1,2)\le\lambda(1,3) <\dots<\lambda(1,2l)=\lambda(1,2l+1)=n<\dots$$

This shows the existence of $2l-2$ eigenvalues (counted with multiplicity) of the Laplace operator smaller than $n$. 
A direct verification shows that $r(t)$ satisfies the equation $K_{\Delta,2}(r)+nr=0$. Since $r(t)$ is positive, then $\lambda(2,1)=n$. Therefore all the eigenvalues of the Laplace operator smaller than $n$ comes from the operators $K_{\Delta,1}$. We conclude that there are exactly $2l-2$ positive eigenvalues of the Laplacian smaller than $n$. Let us show the inequality for the stability index. A direct verification shows that the function $r^\prime$ satisfy the equation $K_{J,1}(r^\prime)=0$. The numbers of zeroes of $r^\prime$ in the interval $[0,mT)$ is  $2m$. By Theorem \ref{zeroes} we conclude that either $0=\tilde{\lambda}(1,2m)$ or $0=\tilde{\lambda}(1,2m+1)$ Therefore the operator $K_{J,1}$ has at least $2m-1$ negative eigenvalues.  Let us analyze the spectrum of the operator $K_{J,2}$.   A direct verification shows that the functions $f_3(t)= \frac{-c^{-n/2} r^{1-n}\cos(\theta)+r r^\prime \sin(\theta)}{\sqrt{1-r^2(t)}}$ and $f_4(t)=\frac{c^{-n/2} r^{1-n}\sin(\theta)+rr^\prime \cos(\theta)}{\sqrt{1-r^2(t)}}$ satisfy the equation $K_{J,2}(f_i)=0$ for  $i=3,4$. The numbers of zeroes $f_3$ and $f_4$ in the interval $[0,mT)$ is at least $2l$ since every time $\theta(t)$ changes from $s\pi$ to $(s+1) \pi$ with $s$ an integer, there is at least a value of $t$ that solves the equation $\cot(\theta(t))=\frac{r^\prime}{c^{n/2}r^n}$. By Theorem \ref{zeroes} we conclude that $0=\tilde{\lambda}(2,2k)=\tilde{\lambda}(2,2k+1)$ where $2k$ is the number of zeroes of $f_4$ and $f_5$ on the interval $[0,mT)$. Recall that $k\ge l$. Therefore the operator $K_{J,2}$ has at least $2l-1$ negative eigenvalues. Since the multiplicity of the eigenvalue $\alpha_2=(n-1)$ of the Laplace operator on $S^{n-1}$ is $n$, we conclude that $\mathrm{ind}(\mathcal{M})\ge (2l-1) n+(2m-1)$.
\end{proof}

\begin{rem}
There is a geometrical reason for the functions $r^\prime$ and $f_i$, $i=1,2,3,4$ to satisfy the respective Hill's equation mentioned in the proof of previous theorem. If we write the immersion $\phi$ in Equation (\ref{theimm}) as $(\phi_1,\dots,\phi_{n+2})$ and the Gauss map $\nu$ in Equation (\ref{thegauss}) as $(\nu_1,\dots,\nu_{n+2})$. The functions $f_1$ and $f_2$ are the last two coordinates of the immersion $\phi$ and, in general, all the coordinates of an immersion of a minimal hypersurface of $\mathbb{S}^{n+1}$ are eigenfunctions of the Laplacian associated with the eigenvalue $n$. The function $r^\prime$ is the simplification of the expression $\phi_{n+2}\nu_{n+1}-\phi_{n+1}\nu_{n+2}$ and in general all the functions of the form $\nu_i\phi_j-\nu_j\phi_i$ are either the zero function or they are eigenfunctions of the stability operator associated with the eigenvalue $0$. Finally, we have that $f_4$ and $f_5$ are factors of the expression $\phi_1\nu_{n+1}-\phi_{n+1}\nu_1$ and $\phi_1\nu_{n+2}-\phi_{n+2}\nu_1$. 
\end{rem}

\begin{thm}\label{indexcmc}
Let $M\subset \mathbb{S}^{n+1}$ be a rotational CMC compact hypersurface. If $m$ and $l$ are the relative prime integers in equation (\ref{p}) then, the number of negative eigenvalues of the Jacobi operator counted with multiplicity is at least $ (2l-1)n+(2m-1)$

\end{thm}

\begin{proof} The proof is similar to the one presented for the minimal case. A direct verification shows that $J_{J,0}[r^\prime]$ vanishes and the numbers of zeroes of $r^\prime$ in the interval $[0,mT)$ is  $2m$. By Theorem \ref{zeroes} we conclude that either $0=\tilde{\lambda}(1,2m)$ or $0=\tilde{\lambda}(1,2m+1)$. Therefore the operator $K_{J,1}$ has at least $2m-1$ negative eigenvalues.  Let us analyze the spectrum of the operator $K_{J,2}$.   A direct verification shows that the functions 

$$f_1(t)= \frac{ c^{-\frac{n}{2}} r(t) \left(c^{n/2} (H \cos (\theta (t))-r^\prime(t) \sin (\theta (t)))+r(t)^{-n} \cos (\theta (t))\right)}{\sqrt{1-r(t)^2}}$$

 and 
 
 $$f_2(t)=\frac{ c^{-\frac{n}{2}} r(t) \left(c^{n/2} (r^\prime(t) \cos (\theta (t))+H \sin (\theta (t)))+r(t)^{-n} \sin (\theta (t))\right)}{\sqrt{1-r(t)^2}}$$ 
 
 satisfy the equation $K_{J,2}(f_i)=0$ for  $i=1,2$. The numbers of zeroes $f_1$ and $f_2$ in the interval $[0,mT)$ is at least $2l$. By Theorem \ref{zeroes} we conclude that $0=\tilde{\lambda}(2,2k)=\tilde{\lambda}(2,2k+1)$ where $2k$ is the number of zeroes of $f_1$ and $f_2$ on the interval $[0,mT)$. Recall that $k\ge l$. Therefore the operator $K_{J,2}$ has at least $2l-1$ negative eigenvalues. Since the multiplicity of the eigenvalue $\alpha_2=(n-1)$ of the Laplace operator on $S^{n-1}$ is $n$, we conclude that the number of negative eigenvalue of the operator $J$ is greater than $(2l-1) n+(2m-2)$.
\end{proof}

%THE EXAMPLE

\section{An example that illustrates the method.}

In this section we pick an explicit rotational $3$-dimensional minimal hypersurface $M\subset \mathbb{S}^4$ and we compute the first three eigenvalues of the Laplacian and its stability index. 

\subsection{Construction of the particular example.} By the Intermediate Value Theorem, it is easy to see that there is a value of $c$ near $ac=2.82842479911$ such that the function $r(t)$  has period $T$ near $aT=2.6722005616$ and $\theta(T)=\frac{4\pi}{3}$. See equation (\ref{p}). In this case $l=2$, $m=3$ and  our manifold $M$ is defined by this choice of $c$. Recall that the immersion $\phi:\mathbb{S}^2\times \mathbb{R}\longrightarrow \mathbb{S}^{4}$ is given by 

\begin{eqnarray}
\phi(y,t)=(r(t)\, y,\sqrt{1-r(t)^2} \cos(\theta(t)),\sqrt{1-r(t)^2} \sin(\theta(t)))
\end{eqnarray}

Figure \ref{graphrtheta} shows the solution $r(t)$ that produces the manifold $M$, and Figure \ref{profilen3} shows the profile curve of the rotational manifold $M$.

%image 0
\begin{figure}[hbtp]
\begin{center}\includegraphics[width=.4\textwidth]{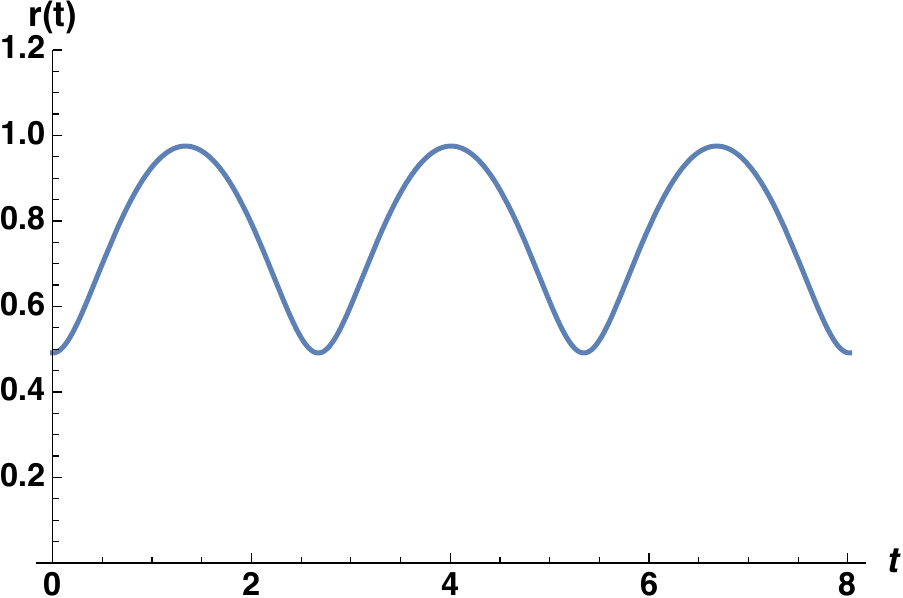} \hskip.25cm \includegraphics[width=.5\textwidth]{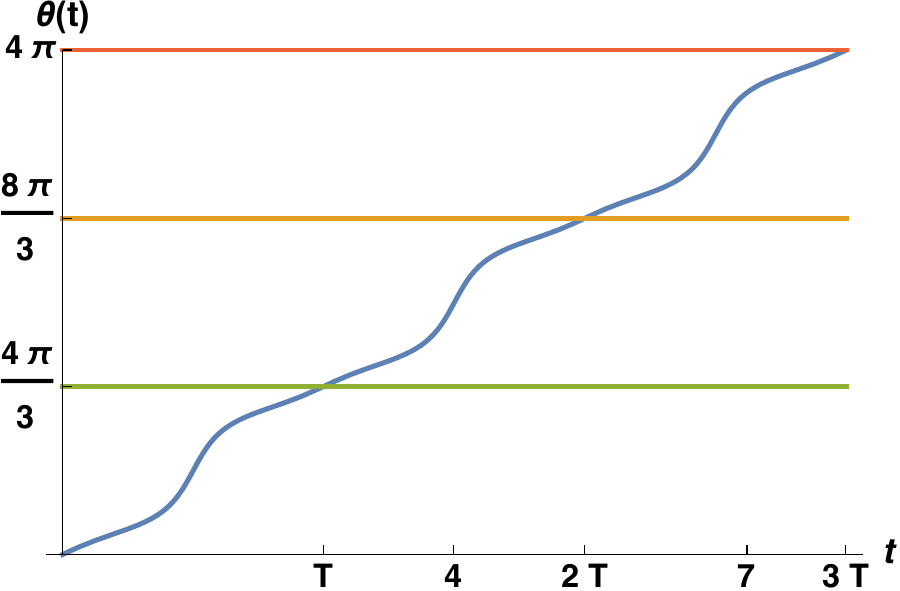}
\caption{On the left we have the graph of the solution $r(t)$ for the value $ac=2.8284247911397589$.  On the right we have the graph of the function $\theta(t)$ defined on Equation (\ref{rel}).}\label{graphrtheta}
\end{center}
\end{figure}
%end image 0

%image 1
\begin{figure}[hbtp]
\begin{center}\includegraphics[width=.5\textwidth]{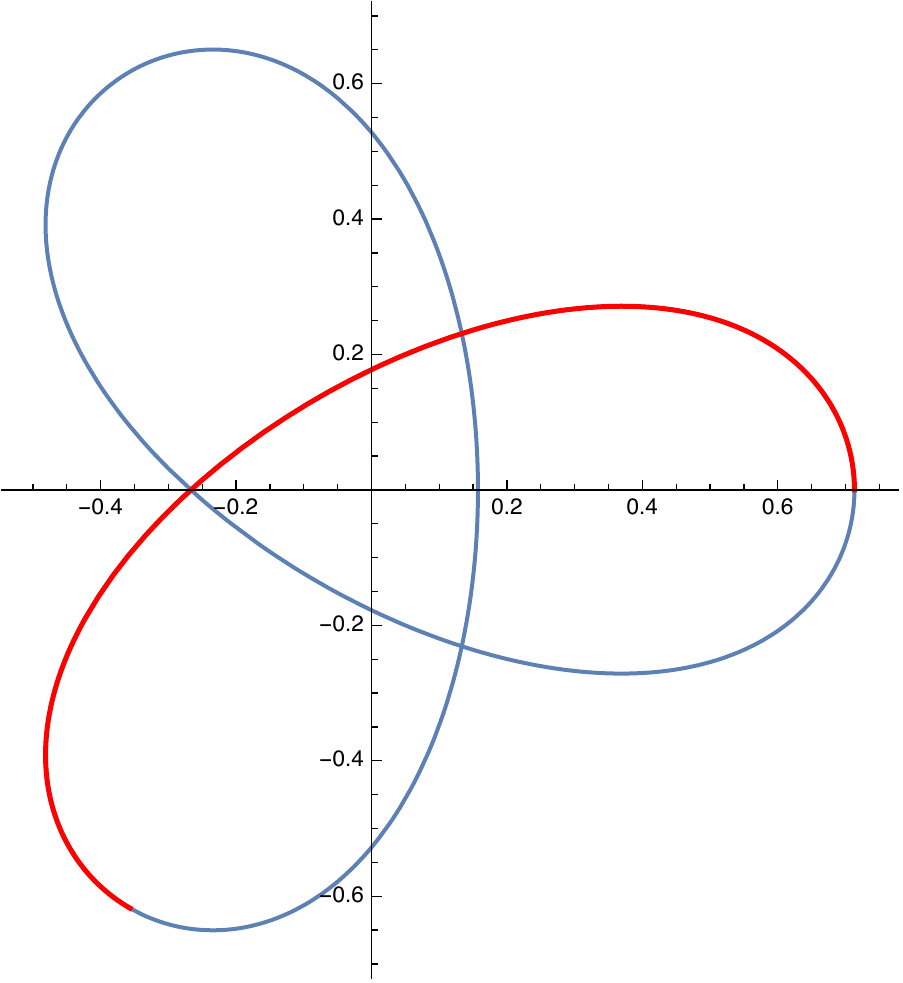} 
\caption{Profile curve of the $M$. This curve is parametrized by $(\sqrt{1-r(t)^2} \cos(\theta(t)),\sqrt{1-r(t)^2} \sin(\theta(t)))$. The red piece represent the curve when $t$ moves from $0$ to $T$. Recall that the period of the immersion is $3T$.}\label{profilen3}
\end{center}
\end{figure}
%end image 1

%subsection: Laplacian.

\subsubsection{Computing the first three eigenvalues of the Laplacian} In order to use the Oscillation Theorem (Theorem \ref{ot}) we notice that making $u=rz$ we obtain  

$$K_{\Delta,k}[z]=z^{\prime\prime}+2\frac{r^\prime}{r} z^\prime - \frac{\alpha_k}{r^2}\,  z =
\frac{1}{r}\left(u^{\prime\prime} + \left(1-\frac{2}{c^3r^6}-\frac{\alpha_k}{r^2}\right)\, u\right)\, .$$

Therefore $\lambda(k,i)$ is an eigenvalue of the operator $K_{\Delta,k}$ if and only if $\lambda(k,i)$ is an eigenvalue of the operator $\bar{K}_{\Delta,k}[u]= u^{\prime\prime} + \left(1-\frac{2}{c^3r^6}-\frac{\alpha_k}{r^2}\right)\, u$. For $\alpha_1=0$, Figure \ref{DeltaAlpha0} shows the discriminant function  $\delta(\lambda)$ for the  operator $\bar{K}_{\Delta,1}$.

%image 2
\begin{figure}[hbtp]
\begin{center}\includegraphics[width=.6\textwidth]{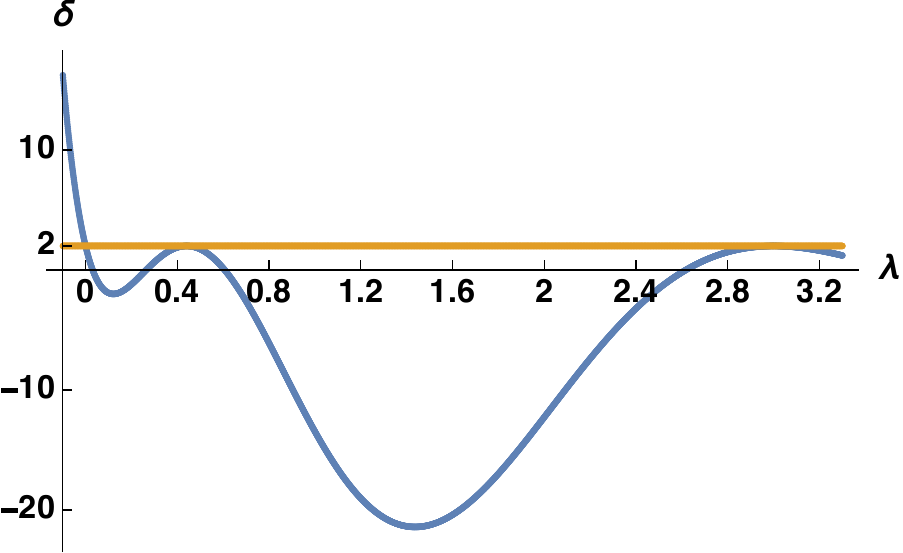} 
\caption{Graph of the function $\delta(\lambda)$. The roots of the equation  $\delta(\lambda)=2$ give us eigenvalues of the Laplacian of the form $\lambda(1,j)$  defined on  Theorem \ref{mt}.}\label{DeltaAlpha0}
\end{center}
\end{figure}
%end image 2

Figure \ref{DeltaAlpha0} was made by taking 3400 values of $\lambda$ between $-0.1$ and $3.3$, one every $0.001$. For each value of $\lambda$ we solve two differential equations to find the functions $z_1(t,\lambda)$ and $z_2(t,\lambda)$ defined in theorem \ref{ot}. Once we have $z_1(t,\lambda)$ and $z_2(t,\lambda)$ we computed $\delta(\lambda)$. The crossing of the graph of $\delta(\lambda)$ with the horizontal line $y=2$ at $\lambda(1,1)=0$ was expected since $z(t)=1$ is an eigenfunction and the crossing at $\lambda=3$ with multiplicity 2 was expected because the last two coordinates of the immersion, the functions $\sqrt{1-r^2}\cos(\theta)$ and $\sqrt{1-r^2}\sin(\theta)$, are eigenfunctions of the Laplacian of $M$, see \cite{S}. Regarding the crossing near $0.44$ we can check that $|\delta(0.4404)-2|$ is smaller than $10^{-6}$. Figure \ref{ef12} shows two linearly independent solutions $\xi_1$ and $\xi_2$ of the equation $K_{\Delta,1}[z]+0.4404 z=0$. All together we have $3$ eigenvalues of $K_{\Delta,1}$ smaller than $3$, which agrees with Theorem \ref{index}.

%image 3
\begin{figure}[hbtp]
\begin{center}\includegraphics[width=.4\textwidth]{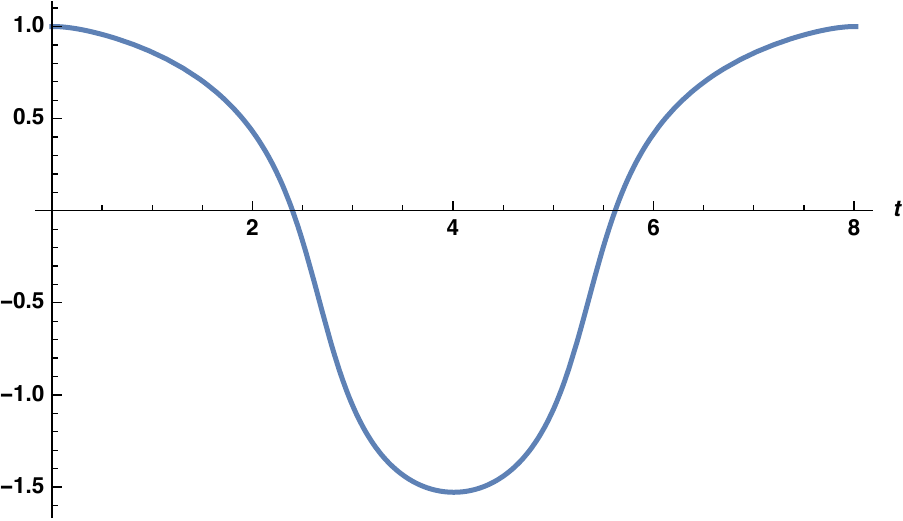}\hskip.9cm \includegraphics[width=.4\textwidth]{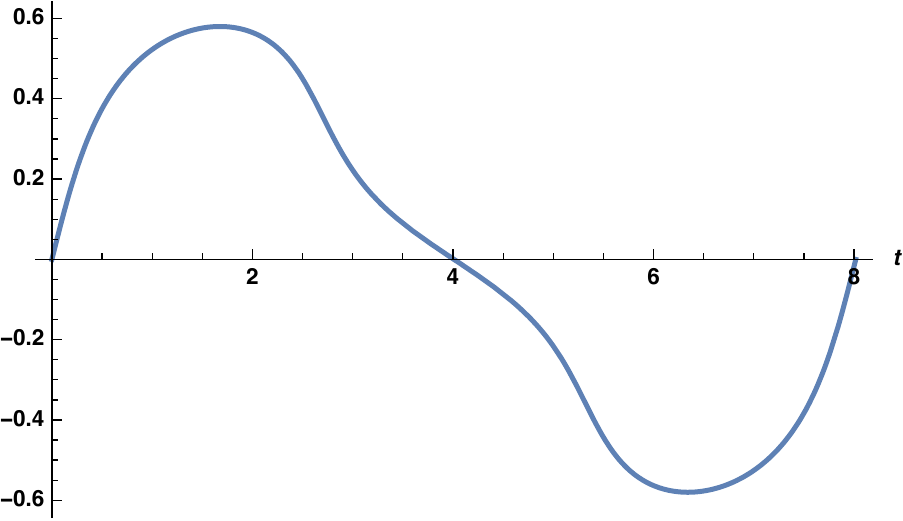} 
\caption{Graph of two solutions  $\xi_1$ and $\xi_2$ of the equation $K_{\Delta,1}[z]+0.4404 z=0$. }\label{ef12}
\end{center}
\end{figure}
%end image 3

We will move now to study the operator  $K_{\Delta,2}$. Since the coordinates of the immersion $\phi$ are eigenfunction of the Laplacian we have that the function $r(t)$ satisfies the equation $K_{\Delta,2}(r(t))=-3 r(t)$. The previous equation also follows from Equation (\ref{soder}). Since $r(t)$ is positive then $\lambda(2,1)=3$ is the first eigenvalue of $K_{\Delta,2}$ and it has multiplicity 1.

\begin{rem}
Since the sequence $\alpha_k$ is increasing then the sequence $\lambda(k,1)$ is also increasing.
\end{rem}

From the previous remark we deduce that all other eigenvalues of the Laplacian of $M$ are greater than $3$. 

%THE EIGENVALUES OF THE JACOBIAN

\subsubsection{Computing the negative eigenvalues of the Jacobi operator} Once again we use the Oscillation Theorem (Theorem \ref{ot}). The change of variables $u=rz$ gives us  

$$K_{J,k}[z]=z^{\prime\prime}+2\frac{r^\prime}{r} z^\prime +\left(\frac{6}{c^3r^6}+3-\frac{\alpha_k}{r^2}\right)\,  z =
\frac{1}{r}\left(u^{\prime\prime} + \left(4+\frac{4}{c^3r^6}-\frac{\alpha_k}{r^2}\right)\, u\right)$$

Similar to the case of the Laplacian operator, we can compute the eigenvalues of the Jacobi operator by computing the eigenvalues of the operator 

$$\bar{K}_{J,k}[u]=
u^{\prime\prime} + \left(4+\frac{4}{c^3r^6}-\frac{\alpha_k}{r^2}\right)\, u\, .$$ 

Figure \ref{JacAlpha0}  shows the discriminant $\delta$ for the  operator $\bar{K}_{J,1}$.  A closer look of the function tell us that the negative solutions of the equation $\delta(\lambda)=2$ are on the intervals $[-8.7,-8.5]$ and $[-3.1, 0]$.

%image 4
\begin{figure}[hbtp]
\begin{center}\includegraphics[width=.3\textwidth]{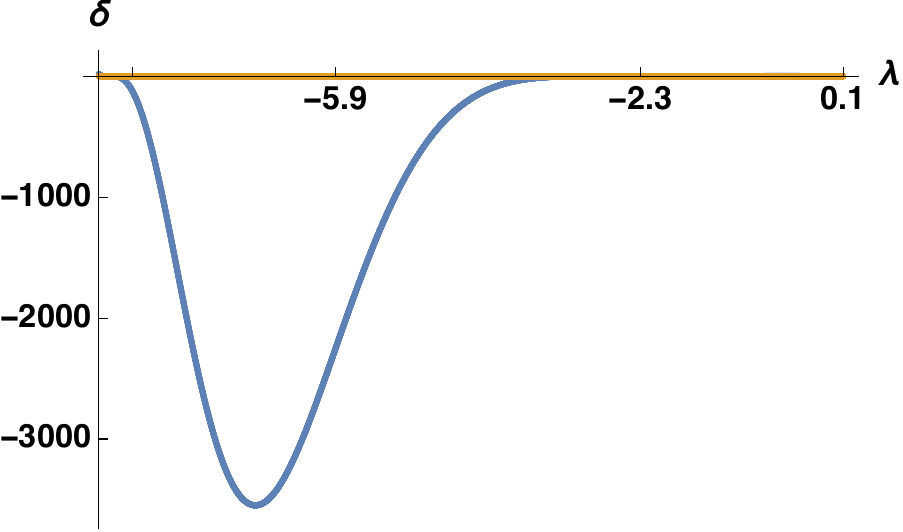} \includegraphics[width=.3\textwidth]{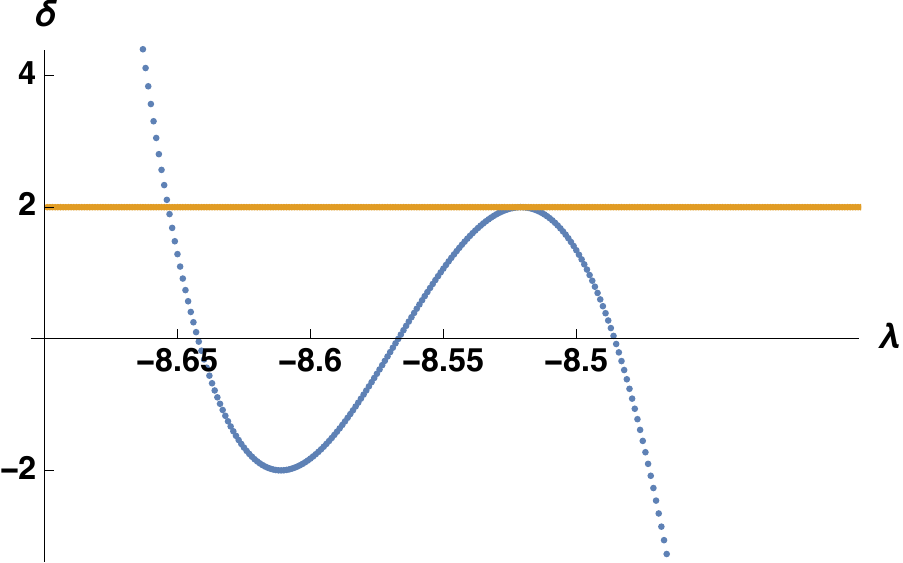}\includegraphics[width=.3\textwidth]{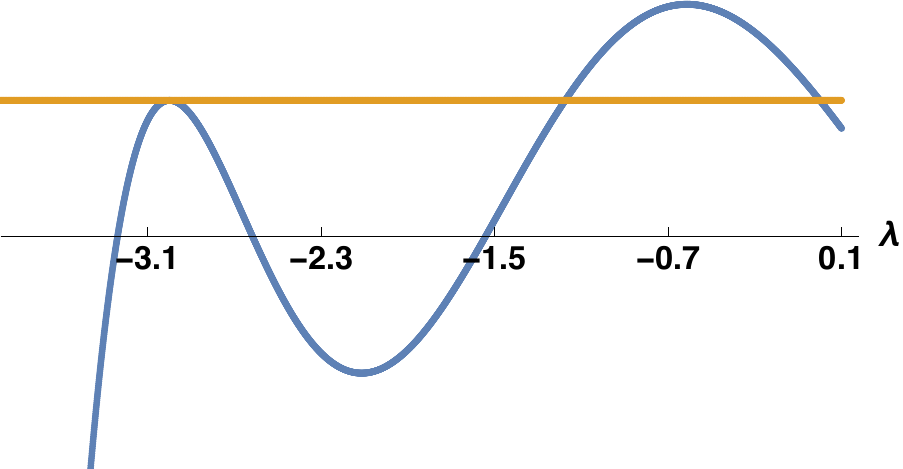} 
\caption{Graph of the function $\delta(\lambda)$ for the operator $\bar{K}_{J,1}$. The roots of the equation  $\delta(\lambda)=2$ give us eigenvalues of the Jacobi operator of the form $\lambda(1,j)$. The graph on the center and on the right shows the function $\delta$ on smaller intervals}\label{JacAlpha0}
\end{center}
\end{figure}
%end image 4

For the first eigenvalue of the Jacobi operator, it is easy to use the Intermediate Value Theorem to show that $\tilde{\lambda}(1,1)=-8.6534$ within an error of $10^{-4}$. This eigenvalue has multiplicity one and Figure \ref{grho} shows a non zero periodic eigenfunction of the operator $K_{J,1}$. For the next value we have that $|\delta(-8.53078)-2|<10^{-5}$. For this value of $\lambda$ the two fundamental solutions of the equation $K_{J,1}[z]+\lambda z=0$ are shown in Figure \ref{f2J}. The next eigenvalue is $-3$ with multiplicity $2$, this eigenvalue was expected due to the fact that the coordinate functions of the Gauss map are eigenfunctions of the Jacobi operator.   For the next eigenvalue we have that $|\delta(-1.1749673)-2|<10^{-5}$. The existence of an eigenvalue near $-1.17496$ with multiplicity one is given by the Intermediate Value Theorem, see Figure \ref{JacAlpha0}. We know that this is the last negative eigenvalue because $0$ is known to be eigenvalue of $K_{J,1}$. 

%image 5
\begin{figure}[hbtp]
\begin{center}\includegraphics[width=.6\textwidth]{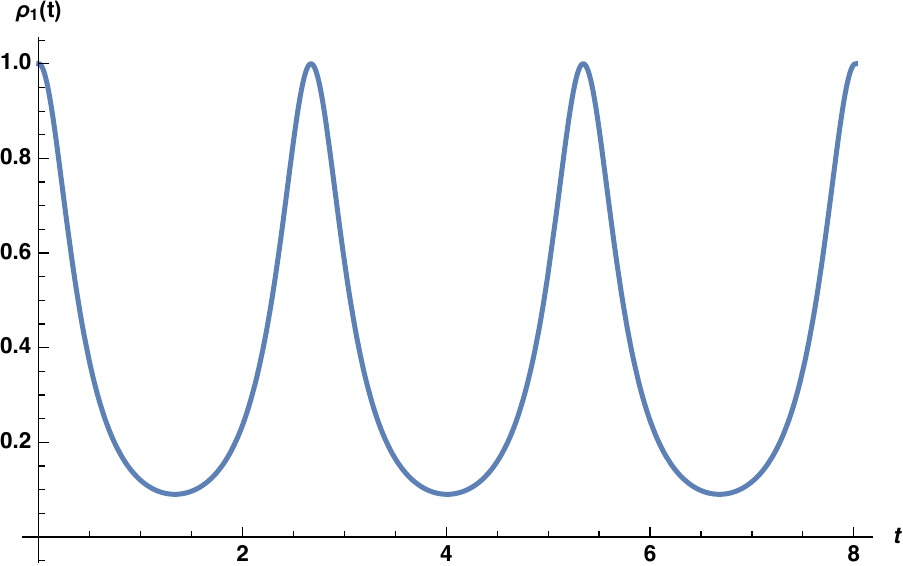} 
\caption{Graph of an eigenfunction associated with $\tilde{\lambda}(1,1)=-8.65...$. This function also represents the first eigenfunction of the stability operator. }\label{grho}
\end{center}
\end{figure}
%end image 5

%image 6  f2J
\begin{figure}[hbtp]
\begin{center}\includegraphics[width=.4\textwidth]{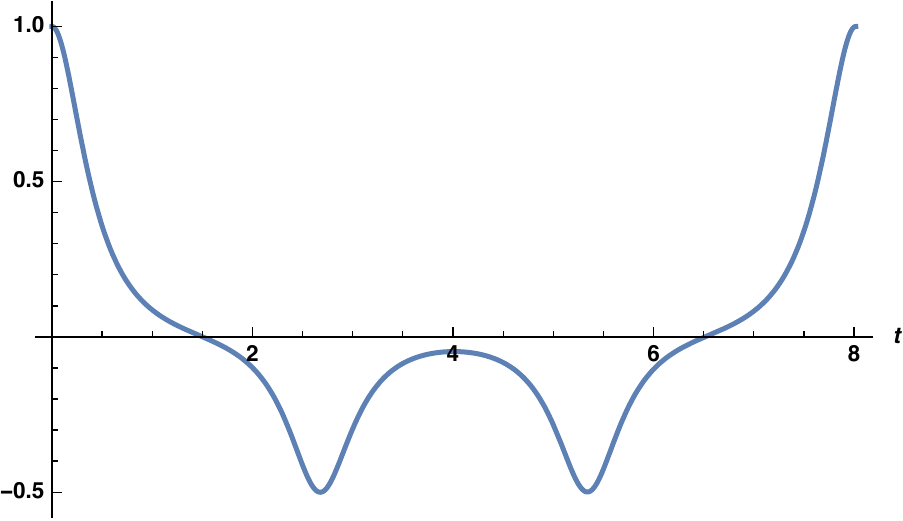}\hskip.4cm\includegraphics[width=.4\textwidth]{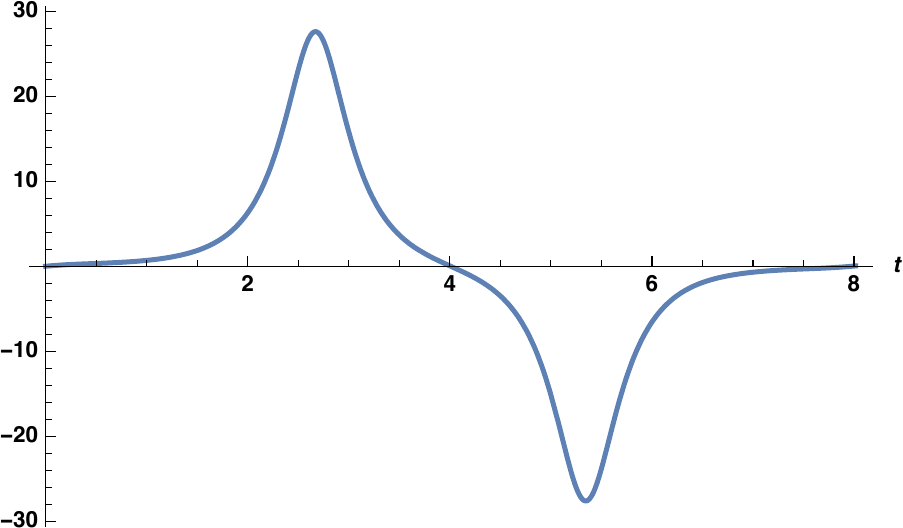} 
\caption{Two solutions of the equation $K_{J,1}[z]-8.53078 z=0$. }\label{f2J}
\end{center}
\end{figure}
%end image 6

We now study the operator $K_{J,2}$. Figure \ref{JacAlpha2} shows the discriminant $\delta$ for the  operator $\bar{K}_{J,2}$. We can directly check that $K_{J,2}(r^{-2})=3 r^{-2}$. Since $r(t)$ is positive, then we have that $-3$ is the first eigenvalue of $K_{J,2}$ with multiplicity one. Since we have that $|\delta(-2.5596)-2|<10^{-5}$, then there is an eigenvalue of $K_{J,2}$ with multiplicity $2$  near $-2.5596$. We can check that the first eigenvalue of the operator $K_{J,3}$ is close to $4.3484453$. Therefore we have gotten all negative eigenvalues of the Jacobi operator, in summary we have

\begin{rem}
Since the first two  eigenvalues of $\mathbb{S}^2$ are $0$ with multiplicity 1 and $2$ with multiplicity $3$, then the stability index of $M$ is 15. Counting with multiplicity we have that the first eigenvalue of the stability operator $J$ on $M$ is near $-8.6534$ and has multiplicity one. We have two eigenvalues near $-8.52078$, it could be only one with multiplicity 2. We have $-3$ with multiplicity $5$. Even though this was known, it is interesting to point out that the multiplicity is $5$ because $-3$ is an eigenvalue with multiplicity 2 of $K_{J,1}$ and $-3$ is an eigenvalue of multiplicity $1$ of $K_{J,2}$, this multiplicity one needs to be multiply by 3 because the eigenvalue  $\alpha_2=2$ of the Laplace operator on $\mathbb{S}^2$ has multiplicity 3. The next eigenvalues are six near $-2.5596$, they are either two with multiplicity 3 or one with multiplicity 6. The last negative eigenvalue is one near $-1.17496$.
\end{rem}

%image 6  JacAlpha 3
\begin{figure}[hbtp]
\begin{center}\includegraphics[width=.6\textwidth]{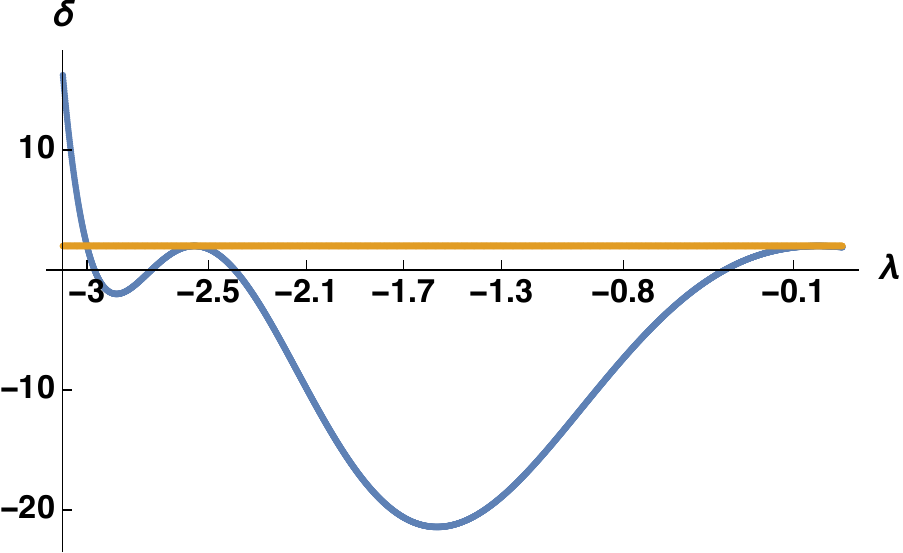}
\caption{Graph of the discriminant of the operator $K_{J,2}$ }\label{JacAlpha2}
\end{center}
\end{figure}
%end image 6

%BIBLIOGRAPHY


\begin{thebibliography}{1}


\bibitem{BL}  Brigitte Beekmann and Hermann Lokes  \emph{Estimate for the Eigenvalues of Hill's equation and Applications for the eigenvalues of the Laplacian on toroidal surfaces}, manuscripta mathematica. {\bf 68},   (1990), 295-308. 


\bibitem{CW}  Hyeong In Choi, Ai-Nung Wang  \emph{A first eigenvalue estimate for minimal hypersurfaces}, J. Diff. Geom. {\bf 18} (1983), 559-562. 

\bibitem{H}  Haupt, O.   \emph{\"Uber eine Methode zum Beweis von Oszillationstheoremen}, Math. Ann. {\bf 76} (1914), 67-104. 

\bibitem{MG}  Marcus Marrocos, Jose Gomes  \emph{Generic specrum of warped products and $G$-manifolds},  ArXiv 1804.02726v2 [math.DG] 27 Oct 2018.


\bibitem{MW}  W. Magnus, S. Winkler  \emph{Hill's equation}, Interscience Publishers, a division of John Wiley \& Sons New York, London, Sydney, 1966.

\bibitem{MN}  Fernando C. Marques, Andr\'e Nevex.  \emph{Min-Max theory and the Willmore conjecture}, Annals of Mathematics {\bf 179} (2014), 683-782.

\bibitem{O1} T. Otsuki  \emph{On integral inequalities related with a certain non linear differential equation}, Proc. Japan Acad. {\bf 48},   (1972),  9-12.

\bibitem{O2} T. Otsuki  \emph{On a differential equation related with differential geometry}, Mem. Fac. Sci. Kyushu Univ. {\bf 47},   (1993),  245-281.

\bibitem{P} Oscar Perdomo.  \emph{Embedded constant mean curvature on spheres}, Asian J. of Math. {\bf 14},  No 1, (2010),  073-108.

\bibitem{P1} Oscar Perdomo.  \emph{Low index minimal hypersurfaces of spheres.} Asian J. Math {\bf 5}, (2001), no 4,  741-749.

\bibitem{P2} Oscar Perdomo.  \emph{On the average of the scalar curvature of minimal hypersurfaces of spheres with low stability index}, Illinois J. of Math. {\bf 48}, (2004),   68-105.


\bibitem{P3} Oscar Perdomo.  \emph{First stability eigenvalue characterization of Clifford hypersurfaces}, Proc. Amer. Math. Soc. {\bf 130}, (2002),   3379-3384.

\bibitem{P4} Oscar Perdomo.  \emph{Rigidity of minimal hypersurfaces of spheres with two principal curvatures}, Arch. Math. (Basel), {\bf 82:2}, (2004),   pp. 180-184.


\bibitem{RS} Rossman, W., Sultana, N.:  \emph{Morse index of constant mean curvature tori of revolution in the 3-sphere}, Illinos. J.  Math. (Basel), {\bf 51:4}, (2004),   pp. 1329-1340.

\bibitem{SA} Savo, Alexander.  \emph{Index bounds for minimal hypersurfaces of the sphere.}, Indiana Univ. Math. J.  {\bf 59}, (2010),  No. 3,  823-837.

\bibitem{S} James Simons.  \emph{Minimal varieties in Riemannian manifolds }, Annals  of Mathematics. (2) {\bf 88}, (2004),  No. 2,  559-565.


\bibitem{B1}  Bruce Solomon  \emph{The Harmonic Analysis of Cubic Isoparametric Minimal Hypersurfaces I: Dimensions 3 and 6}, American Journal of Mathematics
{\bf 112}, No. 2  (1990), 157-203 

\bibitem{B2}  Bruce Solomon  \emph{The Harmonic Analysis of Cubic Isoparametric Minimal Hypersurfaces II: Dimensions 12 and 12}, American Journal of Mathematics
{\bf 112}, No. 2  (1990), 205-241. 

\bibitem{U} Francisco Urbano \emph{Minimal surfaces with low index in the three dimensional sphere}, Proc. Amer. Math. Soc.  {\bf 108} (1990), 898-992.

\bibitem{Y} Shing Tung Yau.  \emph{Problem section}, Seminar on Differential Geometry, Ann. of Math Stu., {\bf 102}, Princeton Univ. Press , (1982),  073-108.



\end{thebibliography}
\end{document}